\documentclass[12pt]{extarticle}
\usepackage{amsmath, amsthm, amssymb, color}
\usepackage[colorlinks=true,linkcolor=blue,urlcolor=blue]{hyperref}
\usepackage{graphicx}
\usepackage{caption}
\usepackage{mathtools}
\usepackage{enumerate}
\usepackage{verbatim}
\usepackage{tikz,tikz-cd,tikz-3dplot}
\usepackage{amssymb}
\usetikzlibrary{matrix}
\usetikzlibrary{arrows}
\usepackage{array}
\usepackage{algorithm}
\usepackage[noend]{algpseudocode}
\usepackage{caption}
\usepackage[normalem]{ulem}
\usepackage{subcaption}
\oddsidemargin \evensidemargin
\usepackage{makecell}
\usepackage{array}
\usepackage{enumitem}

\newtheorem{theorem}{Theorem}[section]
\newtheorem{proposition}[theorem]{Proposition}
\newtheorem{lemma}[theorem]{Lemma}
\newtheorem{corollary}[theorem]{Corollary}

\theoremstyle{definition}
\newtheorem{definition}[theorem]{Definition}
\newtheorem{problem}[theorem]{Problem}
\newtheorem{remark}[theorem]{Remark}

\newtheorem{examplex}[theorem]{Example}
\newenvironment{example}
{\pushQED{\qed}\examplex}
{\popQED\endexamplex}

\makeatletter
\def\Ddots{\mathinner{\mkern1mu\raise\p@
\vbox{\kern7\p@\hbox{.}}\mkern2mu
\raise4\p@\hbox{.}\mkern2mu\raise7\p@\hbox{.}\mkern1mu}}
\makeatother

\newcommand{\cA}{\mathcal{A}}
\newcommand{\cM}{\mathcal{M}}
\newcommand{\cG}{\mathcal{G}}
\newcommand{\cC}{\mathcal{C}}
\newcommand{\cB}{\mathcal{B}}
\newcommand{\bP}{\mathbb P}
\newcommand{\bR}{\mathbb{R}}
\newcommand{\fx}{\mathbf{x}}

\usepackage{enumitem}
\usepackage{array}

\newcommand{\Gr}{\text{Gr}}

\newcommand{\alt}{\mathrm{alt}}
\newcommand{\var}{\mathrm{var}}

\title{\bf Combinatorics of $m=1$ Grasstopes}

\author{Yelena Mandelshtam,
  Dmitrii Pavlov, and Elizabeth Pratt}
\date{}
\begin{document}
\maketitle

\begin{abstract}
\noindent
A Grasstope is the image of the totally nonnegative Grassmannian $\Gr_{\geq 0}(k,n)$ under a linear map $\Gr(k,n)\dashrightarrow \Gr(k,k+m)$. This is a generalization of the amplituhedron, a geometric object of great importance to calculating scattering amplitudes in physics. The amplituhedron is a Grasstope arising from a totally positive linear map. While amplituhedra are relatively well-studied, much less is known about general Grasstopes. We study Grasstopes in the $m=1$ case and show that they can be characterized as unions of cells of a hyperplane arrangement satisfying a certain sign variation condition, extending the work of Karp and Williams \cite{KarpWilliams}. Inspired by this characterization, we also suggest a notion of a Grasstope arising from an arbitrary oriented matroid.
\end{abstract}

\section{Introduction} \label{sec:intro}

The (tree) amplituhedron, introduced by Arkani-Hamed and Trnka in \cite{AHT}, is a geometric object playing an important role in calculations of scattering amplitudes in planar $N=4$ Super-Yang-Mills theory. 
It is defined as the image of the totally nonnegative Grassmannian $\Gr_{\geq 0}(k,n)$ under a totally positive linear map $\Tilde{Z}:\Gr(k,n)\dashrightarrow \Gr(k,k+m)$ given by an $n\times(k+m)$ matrix $Z$. 
While immediate physical relevance of the amplituhedron becomes manifest at $m=4$, it is an object of independent mathematical interest for any $m$. It is known that when $k=1$, it is a cyclic polytope \cite{bernd} and when $k+m=n$, it is isomorphic to the totally nonnegative Grassmannian $\Gr_{\geq 0}(k,n)$.

In recent years, the amplituhedron has been studied extensively from the point of view of algebraic combinatorics for $m=1,2,4$ (see \cite{bcfw, parity, shelling, KarpWilliams, meq2pos, meq2}). 
The structure of the amplituhedron in the $m=1$ case is particularly simple: Karp and Williams \cite{KarpWilliams} show that it is linearly homeomorphic to the complex of bounded cells of an affine hyperplane arrangement and therefore is homeomorphic to a closed ball. 

One reason that the amplituhedron is so amenable to combinatorial study is that the totally nonnegative Grassmannian has a rich combinatorial structure. 
In particular, $\Gr_{\geq 0}(k,n)$ admits a stratification by positroid cells, which are all homeomorphic to open balls \cite{postnikov}. 
However, amplituhedra are images of very special linear maps, just as cyclic polytopes are very special polytopes. From this point of view, it makes sense to consider images of the totally nonnegative Grassmannian under arbitrary linear maps. 
In \cite{Lam16:14:CDM}, Lam considers the images of positroid cells under arbitrary linear maps, and calls them \emph{Grassmann polytopes}. 
Images of the whole totally nonnegative Grassmannian are referred to as \emph{full Grassmann polytopes}. 

While amplituhedra have attracted significant attention from the mathematical community, many results in this area rely on the total positivity of the map $\Tilde{Z}$, and much less is known about more general Grassmann polytopes.
Most of the known results are assembled in \cite[Part 2]{Lam16:14:CDM}. Many of them are limited to the case when $\Tilde{Z}$ does not have base points on the totally nonnegative Grassmannian, i.e. is regular (well-defined) on $\Gr_{\geq 0}(k,n)$. 

In this paper, we initiate the study of general full Grassmann polytopes (\emph{Grasstopes}) and focus on the case of $m=1$, when the ambient Grassmannian is the dual projective space $(\bP^k)^\vee$.
We extend the results of \cite{KarpWilliams}, showing that when $\Tilde{Z}$ is regular on $\Gr_{\geq 0}(k,n)$, the resulting Grasstope is a union of cells of a projective hyperplane arrangement satisfying a certain sign variation condition.
This, in particular, implies that Grasstopes arising from such $\Tilde{Z}$ are closed and connected, in accordance with \cite[Proposition 15.2]{Lam16:14:CDM} (note that this result applies to Grasstopes for general $m$ but only for the tame case). 
When there are no additional restrictions on $\Tilde{Z}$, we show that the image of the totally positive Grassmannian $\Gr_{>0}(k,n)$ can still be characterized in terms of sign changes, although the image of $\Gr_{\geq 0}(k,n)$ might have irregular boundary.
We also show that, unlike amplituhedra, general $m=1$ Grasstopes are not necessarily homeomorphic to closed balls or even contractible.

This paper is organized as follows.
In Section \ref{sec:prelim}, we define Grassmannians, their totally nonnegative and totally positive parts, and Grasstopes. 
We divide Grasstopes into three categories (tame, wild, and rational) based on the properties of the map $\Tilde{Z}$. We note that the terms ``tame'' and ``wild'' for Grasstopes were introduced by Lam.  
We introduce the concepts necessary for the sign variation characterization of $m=1$ Grasstopes and  prove some auxiliary results about general Grasstopes. 
In Section \ref{sec:meq1}, we study the combinatorics and geometry of Grasstopes for $m=1$ and prove the sign variation characterization results for tame and wild Grasstopes, as well as for open rational Grasstopes.
Section \ref{sec:examples} is devoted to examples. In Section \ref{sec:ormatdefs}, we give background on oriented matroids, which is useful for Section \ref{sec:top}. Finally, in Section \ref{sec:top} we investigate how many regions of a hyperplane arrangement can be in an $m=1$ Grasstope, and, based on the sign variation characterization, suggest a definition of the Grasstope of a (not necessarily realizable) oriented matroid. 

\section{Preliminaries} \label{sec:prelim}

The \emph{real Grassmannian} $\Gr(k,n)$ is the variety parameterizing $k$-dimensional subspaces of an $n$-dimensional vector space $\mathbb{R}^n$. 
This can be identified with the variety of $(k-1)$-dimensional subspaces of an $(n-1)$-dimensional real projective space $\bP^{n-1}$. 
A linear space, considered as an element of $\Gr(k,n)$, can be represented by a $k\times n$ matrix $A$ whose rows span the space. 
The Grassmannian $\Gr(k,n)$ can be realized as a subvariety of the projective space $\bP^{\binom{n}{k}-1}$ via the \emph{Pl\"ucker embedding}, which sends a matrix $A$ representing an element in $\Gr(k,n)$ to the vector of its maximal minors. 
Maximal minors of $A$ are called \emph{Pl\"ucker coordinates}. 
This embedding does not depend on the choice of matrix representatives.
For more details on the Grassmannian and the Pl\"ucker embedding, see \cite[Chapter 5]{nlalg}.

The \emph{totally nonnegative Grassmannian} $\Gr_{\geq 0}(k,n)$ is the subset of $\Gr(k,n)$ consisting of the elements whose non-zero Pl\"ucker coordinates all have the same sign \cite{postnikov, Williams:2021zph}. 
Each element in $\Gr_{\geq 0}(k,n)$ can be represented by a \emph{totally nonnegative} $k\times n$ matrix, that is, by a matrix with nonnegative maximal minors.
The \emph{totally positive Grassmannian} $\Gr_{> 0}(k,n)$ is a subset of $\Gr(k,n)$ consisting of the elements whose Pl\"ucker coordinates are all non-zero and have the same sign. 
Elements of $\Gr_{>0}(k,n)$ can be represented by $k\times n$ matrices with strictly positive maximal minors (such matrices are called \emph{totally positive}).

Let~$Z$ be a~real $n\times (k+m)$ matrix of full rank, where~$k+m \leq n$. 
The matrix $Z$ defines a rational map~$\Tilde{Z}: \Gr(k,n)\dashrightarrow \Gr(k,k+m)$ by~$[A] \mapsto [AZ]$, where~$[A]$ is the class in~$\Gr(k,n)$ of a matrix $A$. 

\begin{definition}[Grasstopes]
The image $\Tilde{Z}(\Gr_{\geq 0}(k,n)) \subseteq \Gr(k,k+m)$ is called the $(n,k,m)$-\emph{Grasstope} of $Z$ and is denoted by $\mathcal{G}_{n,k,m}(Z)$. 
\end{definition}

The totally nonnegative Grassmannian $\Gr_{\geq 0}(k,n)$ is a basic semialgebraic set (that is, it can be described by polynomial equations and inequalities) and $\mathcal{G}_{n,k,m}(Z)$ is its image under a polynomial map. Thus, it follows from the Tarski-Seidenberg theorem \cite[Theorem A.49]{soscon} that $\mathcal{G}_{n,k,m}(Z)$ is also semialgebraic.

In \cite{Lam16:14:CDM} $\mathcal{G}_{n,k,m}(Z)$ are called \emph{full Grassmann polytopes}.
When $Z$ is a totally positive matrix, one recovers the definition of the \emph{(tree) amplituhedron} $\mathcal{A}_{n,k,m}(Z)$ \cite{AHT}, a geometric object of fundamental importance to calculating scattering amplitudes in particle physics. 

Note that the matrix $AZ$ defines an element in $\Gr(k,k+m)$ if and only if $AZ$ has full rank.
It is a priori not guaranteed that the map $\Tilde{Z}$ is \emph{well-defined} on $\Gr_{\geq 0}(k,n)$, that is, that $AZ$ has full rank for any totally nonnegative $k\times n$ matrix $A$.
In general, the map $\Tilde{Z}$ has base locus $$\cB(\Tilde{Z}) := \{  V : \text{dim}(V \cap \ker Z^T) \geq 1 \} \subset \Gr(k,n),$$ which may or may not intersect $\Gr_{\geq 0}(k,n).$ 
Note that the $\cB(\Tilde{Z})$ is a Schubert variety, so in particular, it is closed in $\Gr(k,n)$ \cite[Section 17]{Lam16:14:CDM}. 
We will often view $\Gr_{\geq 0}(k,n)$ as a parameter space of $(k-1)$-dimensional subspaces of $\bP^{n-1},$ in which case the base locus is all projective subspaces which are not disjoint from $\bP(\ker Z^T)$. 

Finding combinatorial conditions for $\Tilde{Z}$ to be well-defined on $\Gr_{\geq 0}(k,n)$ had been an active area of research for several years.
In \cite[Proposition 15.2]{Lam16:14:CDM} Lam proved that if $\Tilde{Z}$ is well-defined on $\Gr_{\geq 0}(k,n)$, then $\mathcal{G}_{n,k,m}(Z)$ is closed and connected. In the same proposition he showed that the following condition is sufficient for $\Tilde{Z}$ to be well-defined.
\begin{equation}\label{cond25}
     \text{There exists a }(k+m) \times k\text{ matrix }M\text{ such that all } k \times k\text{ minors of }ZM \text{ are positive}.
\end{equation}
Geometrically, condition \eqref{cond25} means that the element of $\Gr(k+m,n)$ represented by $Z$ contains a totally positive $k$-dimensional subspace, that is, an element of $\Gr_{>0}(k,n)$.
Lam also conjectured that this condition is necessary for $\Tilde{Z}$ to be well-defined on $\Gr_{\geq 0}(k,n)$. This conjecture turned out to be false, with a counterexample given by Galashin (see \cite[Remark 9.3]{KarpWilliams} and Example \ref{eg:wild}). A combinatorial criterion was given in \cite[Theorem 4.2]{karp}. Lam's conjecture gives rise to the following definition. 

\begin{definition}[Tame Grasstope]
   The Grasstope $\mathcal{G}_{n,k,m}(Z)$ is called \emph{tame} if $Z$ satisfies \eqref{cond25}. 
\end{definition}

Let $\phi: V \to W$ be a map of vector spaces. We define its $k^{\text{th}}$ exterior power $\wedge^k\phi: \bigwedge^kV \to \bigwedge^k W$ by $ v_1 \wedge \ldots \wedge v_k \mapsto \phi(v_1) \wedge \ldots \wedge \phi(v_k).$ If $M$ is a matrix representing $\phi$ in bases $\{e_i\}$ of $V$ and $\{\Tilde{e}_j\}$ of $W$, we denote the matrix representing $\wedge^k \phi$ in the induced bases of $\bigwedge^k V$ and $\bigwedge^k W$ by $\wedge^k M$. 
Then the matrix of $\Tilde{Z}$ is $\wedge^kZ,$ where we use the \emph{Pl\"ucker embedding} $\Gr(k, n) \to \mathbb{P}(\bigwedge^k(\bR^n)), \ \text{Span}\{v_1,\ldots, v_k\} \mapsto v_1 \wedge \ldots \wedge v_k$. 
Concretely, the entries of $\wedge^kM$ are the $k\times k$ minors of $M$, ordered by multi-indices in reverse lexicographic order.

We now give a simple geometric criterion of tameness. Although well-known to specialists in this area, this result seems to have not appeared in the literature yet.  In what follows, for the sake of simplicity, we slightly abuse notation and write $\mathcal{G}_{n,k,m}(Z)$ both for the Grasstope as a subset of~$\Gr(k,k+m)$ and its image under the Pl\"ucker embedding. 

\begin{proposition}
\label{thm:tame}
An $n\times (k+m)$ matrix $Z$ satisfies \eqref{cond25} (i.e. $\mathcal{G}_{n,k,m}(Z)$ is tame) if and only if there exists a hyperplane in $\bP^{\binom{k+m}{k}-1}$, corresponding to a point in $\Gr(k, k+m)$ which does not intersect $\mathcal{G}_{n,k,m}(Z)$. 
\end{proposition}
    \begin{proof}
    Any $(k+m) \times k$ matrix $M$ defines a hyperplane in $\bP^{\binom{k+m}{k} -1}$ by its Pl\"ucker coordinates: 
    a point $p \in \bP^{\binom{k+m}{k} -1}$ lies on the hyperplane defined by $M$ if and only if $(\wedge^k M)^Tp = 0$.
    Suppose that $Z$ satisfies \eqref{cond25} and $M$ is a $(k+m)\times k$ matrix such that $ZM$ is totally positive. Suppose that a point in $\mathcal{G}_{n,k,m}(Z)$ lies on the hyperplane defined by $M$.
    Then, for some $A \in \Gr_{\geq 0}(k, n)$, it holds that $(\wedge^k M)^T(\wedge^k AZ)^T = 0$, which is equivalent to $(\wedge^k(ZM))^T(\wedge^k A)^T = 0$. 
    Since all $k\times k$ minors of $A$ are non-negative, and at least one is nonzero, it is not possible for all $k\times k$ minors of $ZM$ to have the same sign.
    Thus, if there exists $M$ such that $ZM$ has all positive (or negative) $k\times k$ minors, then the image $\mathcal{G}_{n,k,m}(Z)$ does not intersect the hyperplane defined by $M$.
    
    Now suppose that $Z$ does not satisfy \eqref{cond25}, that is, for any $M$ the matrix $ZM$ has either a zero $k\times k$ minor or at least one positive and one negative $k\times k$ minor.
    We will show that there exists a matrix $A \in \Gr_{\geq 0}(k, n)$ such that $(\wedge^k(ZM))^T(\wedge^k A)^T = 0$, so that the hyperplane defined by $M$ intersects $\Tilde{Z}(\Gr_{\geq 0}(k, n))$.
    
    In the first case, when $ZM$ has a zero minor in the $i^{\text{th}}$ Pl\"ucker coordinate, one can find an element $A \in \Gr_{\geq 0}(k, n)$ which has all Pl\"ucker coordinates equal to zero except for the $i^{\text{th}}$ one. 
    In this case, $(\wedge^k(ZM))^T(\wedge^k A)^T = 0$.
    
    Now consider the second case, in which $ZM$ has  all $k\times k$ minors nonzero with at least one positive and one negative. By connectivity of the basis exchange graph of the uniform matroid, there exists a set of column indices $I = \{i_1, \dots, i_{k-1}\}$ such that two Pl\"ucker coordinates involving $I$ (which we label $p_{I\cup \{i\}}$ and $p_{I\cup \{j\}}$) have different signs. 
    Then, take $(q_{1, \dots k}: \dots :q_{n-k+1, \dots, n}) \in \bP^{\binom{n}{k} - 1}$ such that all coordinates except for $q_{I \cup \{i\}}, q_{I\cup\{j\}}$ are zero, and $q_{I\cup\{i\}} = |p_{I\cup\{j\}}|$ and $q_{I\cup \{j\}} = |p_{I\cup\{i\}}|$. 
    Then, all Pl\"ucker relations are satisfied so this point represents an element $A \in \Gr_{\geq 0}(k, n)$. We have $(\wedge^k(ZM))^T(\wedge^k A)^T = 0$, so the hyperplane given by $M$ intersects the Grasstope $\mathcal{G}_{n,k,m}(Z)$.  
    \end{proof}

    Note that when $m=1$, every hyperplane in $\bP^{\binom{k+m}{k}-1}$ corresponds to some point in $\Gr(k, k+m)$. Then by choosing a hyperplane disjoint from $\mathcal{G}_{n,k,1}(Z)$ to be the hyperplane at infinity, we arrive at the following result. 

    \begin{corollary}\label{cor:tame}
    The Grasstope $\mathcal{G}_{n,k,1}(Z)$ is tame if and only if its image under the Pl\"ucker embedding is contained in some affine chart of $\bP^{k}$.
    \end{corollary}

Tame Grasstopes share many nice properties with amplituhedra. 
In particular, for $m=1$ they are homeomorphic to closed balls and can be described as complexes of bounded cells of affine hyperplane arrangements \cite[Section 9]{KarpWilliams}. 
The focus of this paper, however, is to take a step away from the tame case and study Grasstopes that behave somewhat less regularly.   

\begin{definition}[Wild Grasstope]
    If the map $\Tilde{Z}$ is well-defined on $\Gr(k,n)_{\geq 0}$ but $Z$ does not satisfy \eqref{cond25}, then the Grasstope $\mathcal{G}_{n,k,m}(Z)$ is called \emph{wild}.
\end{definition}

Even though $\Tilde{Z}$ might not be well-defined on $\Gr_{\geq 0}(k,n)$, it still makes sense to consider the image of $\Gr_{\geq 0}(k,n)\setminus (\cB(\Tilde{Z}) \cap \Gr_{\geq 0}(k,n))$, where $\cB(\Tilde{Z})$ is the base locus of $\Tilde{Z}$.
In this case we will still write $\Tilde{Z}(\Gr_{\geq 0}(k,n))$ for this image.

\begin{definition}[Rational Grasstope]
    Suppose the map $\Tilde{Z}$ is not well-defined on $\Gr(k,n)_{\geq 0}$. Then the image $\mathcal{G}_{n,k,m}(Z) = \Tilde{Z}(\Gr(k,n)_{\geq 0})$ is called a \emph{rational} Grasstope. The image $\mathcal{G}^\circ_{n,k,m}(Z) = \Tilde{Z}(\Gr(k,n)_{> 0})$ of the totally positive Grassmannian is called an \emph{open rational} Grasstope. From $m=1$, this is indeed an open set, as shown in Proposition \ref{prop:opengrass}.
\end{definition}

We conclude this section with technical results that will prove useful in characterization of $m=1$ Grasstopes in Section \ref{sec:meq1}. 

Given a point $u$ in $\bP^n,$ we associate a sign pattern $\sigma = (\sigma_0,\ldots,\sigma_n) \in \{+,-,0\}^{n+1}$ to $u$ in the following way.
Pick $i$ such that $u_i \neq 0$ and set $\sigma_i := +$.
Then $\sigma_j := \operatorname{sign}(u_iu_j)$, which is a well-defined function of homogeneous coordinates of~$u$.
Since we associate sign labels to points in projective space, we will identify sign labels $\sigma$ and $-\sigma$.
Each orthant of $\bP^n$ consists of the points with the same sign pattern.
For instance, the sign pattern $( + : + : - : + )$ for a point in $\bP^3$ represents the orthant defined by $$\{u_0u_1 > 0, u_0u_2 < 0, u_0u_3 > 0, u_1u_2 < 0, u_1u_3 > 0, u_2u_3 < 0\}.$$ 

Given a point $x \in \bP^k$ being the image of a hyperplane $X = \text{Span}\{w_1, ..., w_k\}$ under the Pl\"ucker embedding of $\Gr(k, k+1),$ and any point $v \in \bP^{k},$ one may consider the bilinear map 
\begin{align*}\label{eq:pairing}
    T : \bR^{k+1} \times \bR^{k+1} & \to \bR \\
    x, v & \mapsto \det 
    \begin{bmatrix}
    \vert & & \vert &  \vert \\
    w_1   &... & w_k  & v \\
    \vert & & \vert & \vert
\end{bmatrix} = \sum_{j=1}^{k+1} (-1)^{k+1-j}p_{1 ... \hat{j} ... k+1}v_j,
\end{align*}
where $x$ and $v$ are represented by any vector of homogeneous coordinates. Then the vanishing of $T$ is a projectively well-defined notion, and $T(x,v) = 0$ if and only if $v$ is contained in $X.$ If $v = Z_i$ for some row $Z_i$ of $Z,$ then $T(x,Z_i)$ is commonly called a \emph{ twistor coordinate of $X$ with respect to $Z$}~\cite[Definition 4.5]{Williams:2021zph}.

\begin{remark} \label{rem:twistor}
    One may also fix either argument $x$ or $v$ of $T$ to get a linear form on $\bP^k.$ In this paper we will consider the  twistor coordinates as functions of $x$ by setting $v=Z_i$ and denote the resulting forms by $l_i(x)$. We will see in Theorem \ref{thm:welldefgrass} that the vanishing loci of the $l_i$'s are exactly the hyperplanes which contain the boundaries of $\cG_{n,k,m}(Z)$.
\end{remark}

We now recall the definition of sign variation from \cite{KarpWilliams}.
\begin{definition}[Sign variation]
    Given a sequence $v$ of $n$ real numbers, let $\var(v)$ be the number of sign changes in $v$ (zeros are ignored).
Let $\overline{\var}(v):=\max\{\var(w):w\in\mathbb{R}^n\text{ such that }w_i =v_i\text{ for all } i\in [n]\text{ with }v_i \neq0\}$, i.e. $\overline{\var}(v)$ is the maximum number of sign changes in $v$ after a sign for each zero component is chosen. Note that both $\var$ and $\overline{\var}$ are well-defined functions of homogeneous coordinates of a point in $\bP^{n-1}$.
\end{definition}

We call a $(k-1)$-dimensional subspace of $\bP^{n-1}$ \emph{totally nonnegative} if it is a point in $\Gr_{\geq 0}(k, n)$ and \emph{totally positive} if it is a point in $\Gr_{> 0}(k, n)$.
Given a point $u \in \bP^{n-1},$ we define the \emph{hyperplane $H_u$ given by $u$} to be the hyperplane orthogonal to $u$ with respect to the standard dot product, i.e.  $H_u := \{v \in \bP^{n-1} : u \cdot v = 0\}.$ Then we have the following proposition, which has appeared in the literature with a different formulation.

\begin{proposition}[\cite{karp}, Lemma 4.1]\label{prop:signchange} Let $H_u$ be the hyperplane given by $u \in \bP^{n-1}.$
\begin{enumerate}
    \item $H_u$ contains a totally nonnegative subspace of dimension $k-1$ if and only if $\overline{\var}(u) \geq k$. 

    \item $H_u$ contains a totally positive susbspace of dimension $k-1$ if and only if $\var(u) \geq k$.
\end{enumerate}
\end{proposition}
We note that the first part of Proposition \ref{prop:signchange} is also equivalent to \cite[Corollary 6.7]{ayala2004control}. 
\begin{proof}
We derive this from \cite[Lemma 4.1]{karp}, which states that, for $v \in \mathbb{P}^{n-1}$,  there exists an element of $\Gr_{\geq 0}(k, n)$ (resp. $\Gr_{>0}(k, n)$) containing $v$ if and only if $\var(v)\leq k -1$ (resp. $\overline{\var}(v) \leq k-1$). We can transform this into the first (resp. the second) statement of our proposition using \cite[Lemma 1.11]{karp} as follows. The hyperplane $H_u$ contains $V \in \Gr_{\geq 0}(k, n)$ if and only if $u \in V^\perp$ for some $V \in \Gr_{\geq 0}(k, n)$, where $V^\perp$ denotes the orthogonal complement of $V$. By \cite[Lemma 1.11, (ii)]{karp} this happens if and only if $\alt(u) \in \alt(V^\perp) \in \Gr_{\geq 0}(n-k, n)$, where $\alt(v) = [v_0: -v_1: v_2 \dots :(-1)^{k-1}v_{k-1}]$ for any $v \in \mathbb{P}^{k-1}$. By \cite[Lemma 4.1]{karp} this is equivalent to $\var(\alt(u)) \leq n-k-1$. Finally, \cite[Lemma 1.11, (ii)]{karp} tells us that $\overline{\var}(u) + \var(\alt(u))=n-1$, so we obtain equivalence with $\overline{\var}(u) \geq k$. The second statement follows similarly, replacing $\Gr_{\geq 0}$ with $\Gr_{>0}$ and $\overline{\var}$ with $\var$ where appropriate.
 
%     We start by proving the first statement. 
%     A hyperplane $H_u$ contains a totally nonnegative subspace of dimension $k$ if and only if $u$ is in the kernel of some matrix of the form $\left( \begin{array}{ccc}
%     & A & \\
%     \hline 
%      &  B & 
% \end{array} \right)$ where $A \in \Gr_{\geq 0}(k, n)$, and $B \in \Gr(n-k-1, n)$. Here $A$ represents the totally nonnegative subspace, and $B$ the additional points to define a hyperplane. Note that $u \in \ker \left( \begin{array}{ccc}
%     & A & \\
%     \hline 
%      &  B & 
% \end{array} \right)$ if and only if $u \in \ker(A)$ and $u \in \ker(B)$. Then, the ``only if" direction follows directly from \cite[Theorems V6 and V7]{GK50} (also \cite[Theorem 3.4(i)]{KarpWilliams}). The ``if" direction also follows, with the note that one can always find $n-k-1$ additional points to form the matrix $B$ such that $u \in \ker(B)$.
% The second statement is proved analogously, with replacing $\Gr_{\geq 0}(k, n)$ by $\Gr_{>0}(k, n)$ and using \cite[Theorem 3.4(ii)]{KarpWilliams}.
\end{proof}

\section{Grasstopes for \texorpdfstring{$m=1$}~: Tame, Wild, and Rational} \label{sec:meq1}

We begin by stating our main theorem which describes any $m=1$ Grasstope $\mathcal{G}_{n, k, 1}(Z)$ arising from a well-defined map $\Tilde{Z}$ as a subset of $\bP^k \cong \Gr(k,k+1)$. This theorem recovers and generalizes some of the results of Karp and Williams describing $m=1$ amplituhedra and tame Grasstopes (\cite[Section 6]{KarpWilliams}).
 
\begin{theorem}\label{thm:welldefgrass}
    Suppose $\Tilde{Z}: \Gr(k, n) \dashrightarrow \Gr(k, k+1)$ is well-defined on $\Gr_{\geq 0}(k, n)$. Then the Grasstope $\mathcal{G}_{n, k, 1}(Z)$ consists of the points $\mathbf{x}\in \bP^k$ such that $\overline{\var}(L(\mathbf{x})) \geq k$, where $L(\mathbf{x})$ is the vector of twistor coordinates of $\mathbf{x}$ with respect to $Z.$
\end{theorem}

\begin{proof}

Let $L(\fx)$ = $(l_1(\fx): ... : l_n(\fx)) \in \bP^{n-1}$ (see Remark \ref{rem:twistor}). By Proposition \ref{prop:signchange}, it suffices to show that $H_{L(\fx)}$ contains a totally nonnegative subspace if and only if $X \in \mathcal{G}_{n,k,1}(Z)$.

For the ``if" direction, suppose that $A \in \Gr_{\geq 0}(k, n)$ and $\fx$ is the vector of Pl\"ucker coordinates of~$[AZ].$ Then for each row $A_i$ of $A,$ we have
$$L(\fx) \cdot A_i = \sum_{j=1}^n T(\fx, Z_j) A_{ij} = T(\fx, \sum_{j=1}^n A_{ij}Z_j) = 0,$$ 
since $\sum A_{ij}Z_j$ is a row of $AZ.$ Thus $H_{L(\fx)}$ contains $A.$

For the ``only if" direction, suppose that $H_{L(\fx)}$ contains a totally nonnegative subspace $A$, and let $v \in \ker Z^T.$ Then 
$$L(\fx) \cdot v = \sum_{i=1}^n T(\fx,Z_i)v_i = T(\fx, \sum_{i=1}^n Z_iv_i) = T(\fx, 0) = 0.$$
So $H_{L(\fx)}$ contains $\bP(\ker Z^T)$. Since $\bP(\ker Z^T) \cap A = \emptyset$ by regularity of $\Tilde{Z}$ on $\Gr_{\geq 0}(k, n),$ we have by dimension considerations that any two hyperplanes containing $A$ and $\bP(\ker Z^T)$ must be equal. However, by the ``if" direction, if $\mathbf{a}$ is the vector of Pl\"ucker coordinates of $[AZ],$ then $H_{L(\mathbf{a})}$ contains $A$. In addition,  $H_{L(\mathbf{a})}$ contains $\bP(\ker Z^T)$. Therefore $H_{L(\fx)} = H_{L(\mathbf{a})}.$ Since $u \mapsto H_u$ is injective, we obtain $\fx = \mathbf{a}.$
    
\end{proof}

Note that in the proof of Theorem \ref{thm:welldefgrass}, the well-defined condition is necessary, since otherwise there might be totally nonnegative subspaces which intersect $\bP(\ker Z^T)$. 
In this case, there is no guarantee that a hyperplane containing a totally nonnegative subspace also contains a totally nonnegative subspace disjoint from $\bP(\ker Z^T)$, which might present problems on the boundary of the Grasstope. 

 However, as we show in the following results, we can still describe the open rational Grasstope in the case that the map $\Tilde{Z}$ is not well-defined on $\Gr_{\geq 0}(k, n)$. We first need the following lemma about totally positive subspaces to show that intersection with $\bP(\ker Z^T)$ does not cause issues.

\begin{lemma} \label{lem:wiggle}
     Let $H$ be a hyperplane in $\bP^{n-1}$ containing an $(n-k-2)$-dimensional subspace $P$. If $H$ contains a totally positive $(k-1)$-dimensional subspace, then it contains a totally positive $(k-1)$-dimensional subspace disjoint from $P$.
\end{lemma}

\begin{proof}
Being a full-dimensional subset of $\mathrm{Gr}(k,n)$ defined by finitely many strict inequalities, the totally positive Grassmannian $\mathrm{Gr}_{> 0}(k,n)$ is an open subset of $\mathrm{Gr}(k,n)$. The set $S_H$ of all $(k-1)$-dimensional subspaces contained in $H$ is a Schubert variety in $\mathrm{Gr}(k,n)$ (in particular, it is closed), and $S_H \cap \mathrm{Gr}_{> 0}(k,n)$ is thus open in $S_H$. The set $S_P$ of $(k-1)$-dimensional subspaces meeting $P$ is also a Schubert variety that intersects $S_H$ non-trivially. The intersection $S_P\cap S_H$ is a closed subvariety of $S_H$. This implies that the complement $(S_H \cap \mathrm{Gr}_{> 0}(k,n)) \setminus (S_H\cap S_P)$ is non-empty, which proves the claim.
\end{proof}

Following Lemma \ref{lem:wiggle}, we are ready to describe the open rational Grasstope $\Tilde{Z}(\Gr_{>0}(k, n))$. Recall that $\cB(\Tilde{Z})$ denotes the set of base points of $\Tilde{Z}$.

\begin{proposition}\label{thm:opengrass}
    For a map $\Tilde{Z}: \Gr_{\geqslant 0}(k, n) \dashrightarrow \Gr(k, k+1)$ given by a matrix $Z$ the open Grasstope $\Tilde{Z}(\Gr_{>0}(k, n))$ consists of the points  $\mathbf{x}\in\bP^k$ such that $\var(L(\mathbf{x})) \geq k$, where $L(\mathbf{x})$ is the vector of twistor coordinates of $\mathbf{x}$ with respect to $Z.$
\end{proposition}

\begin{proof}
    We will prove this statement by slightly modifying the proof of Theorem \ref{thm:welldefgrass}.  
    Fix a totally positive matrix $A \in \Gr_{>0}(k,n)\setminus \cB(\Tilde{Z})$ and let $\mathbf{x}$ be the vector of Pl\"ucker coordinates of $[AZ].$ Then the hyperplane $H_{L(\mathbf{x})}$ contains $\bP(\ker Z^T)$ and the totally positive subspace $A$. Conversely, if a hyperplane $H$ contains $\bP(\ker Z^T)$ and a totally positive subspace $A$, then by Lemma \ref{lem:wiggle} $H$ contains a totally positive subspace $A'$ disjoint from $\bP(\ker Z^T)$ (thus, $A'\not\in \cB(\Tilde{Z})$).
    The hyperplane is then uniquely determined by its containment of $\bP(\ker Z^T)$ and $A'$, so it must be $H_{L(\mathbf{x})},$ where $\mathbf{x}$ is the vector of Pl\"ucker coordinates of $[A'Z]$.
    We now use the second part of Proposition \ref{prop:signchange} to conclude that the open Grasstope $\Tilde{Z}(\Gr_{>0}(k, n))$ consists of the points in $\mathbf{x}\in\bP^k$ such that $\var(L(\mathbf{x})) \geq k$.
\end{proof}

The following proposition is an analog of \cite[Lemma 9.4]{parity} for $m=1$ Grasstopes.
\begin{proposition}
\label{prop:opengrass}
    The open rational Grasstope of $Z$ is open and the rational Grasstope of $Z$ is contained in the closure of the open rational Grasstope of $Z.$
\end{proposition}
\begin{proof}
    Let $Z$ define a rational Grasstope and consider a point $\mathbf{x}$ with $\var(L(\mathbf{x})) \geq k$, where $L(\mathbf{x})$ is the vector of twistor coordinates of $\mathbf{x}$ with respect to $Z$. Note that each zero row of $Z$ results in an entry of $L(\mathbf{x})$ that is zero for every $\mathbf{x}$. Since deleting zero entries from $L(\mathbf{x})$ does not change $\var$, we may assume w.l.o.g. that $Z$ has no zero rows.  
    Then for all points $\mathbf{x'}$ in a  sufficiently small neighborhood around it, $L(\mathbf{x'})$ has the same signs as $L(\mathbf{x})$ in all the indices of the nonzero entries of $L(\mathbf{x})$. Since changing the zero entries cannot decrease the sign variation, the neighborhood is contained in the open rational Grasstope.

    Similarly to the reasoning of Theorem \ref{thm:welldefgrass}, it follows from the second part of Proposition \ref{prop:signchange} that the rational Grasstope of $Z$ must be contained in the set $C := \{\mathbf{x}: \overline{\var}(L(\mathbf{x})) \geq k$\}. We show that $C$ is the closure of the open rational Grasstope of $Z$. First, we show that the complement of $C$, the set $\{\mathbf{x}: \overline{\var}(L(\mathbf{x})) < k\},$ is open. Let $\mathbf{x}$ be such that $\overline{\var}(L(\mathbf{x})) < k$. Then for all points $\mathbf{x'}$ in a  sufficiently small neighborhood around $\mathbf{x}$, $L(\mathbf{x'})$ has the same signs as $L(\mathbf{x})$ in all the indices of the nonzero entries of $L(\mathbf{x})$. Since changing the values of the zero entries cannot increase $\overline{\var}$, we know $\overline{\var}(L(\mathbf{x'}))< k$. Therefore $C$ is closed. 
    
    Now consider a point $\mathbf{x} \in C$ and an open neighborhood $N$ of points around it. We will show that there is some $\mathbf{x'} \in N$ with $\var(L(\fx')) \geq k$, which is sufficient to conclude that $C$ is the closure of the open rational Grasstope of $Z$. Since we may assume that $Z$ has no zero rows, each zero entry of $L(\mathbf{x})$ corresponds to containment of $\mathbf{x}$ in a hyperplane. Any open neighborhood around $\mathbf{x}$ contains points on either side of the hyperplane. Similarly, if $\mathbf{x}$ lies in the intersection of several hyperplanes, any open neighborhood of $\mathbf{x}$ contains points in each orthant defined by these hyperplanes, so any sign pattern can be achieved in the entries which are zero in $L(\mathbf{x})$. In particular, this means that there is a point $\mathbf{x'}$ with signs in the nonzero entries of $L(\mathbf{x})$ equal to the signs of the corresponding entries of $L(\mathbf{x'})$ (since we can restrict $N$ to be small enough such that the signs in the nonzero entries of $L(\mathbf{x})$ are unchanged), and signs of the zero entries replaced by the signs ensuring that $\var(L(\mathbf{x'})) \geq k$.
    
\end{proof}

 In Section \ref{sec:examples}, we show that the rational Grasstope may be equal to the closure of the open rational Grasstope. However, we do not know if this holds in general, since as far as we can tell, it might be possible for points in the boundary to fall out. Thus the question of fully describing which parts of the boundary are contained in $m=1$ rational Grasstopes remains open.

\section{Examples} \label{sec:examples}

In this section we provide examples of the families of Grasstopes we considered. We begin with an example of a tame Grasstope.

\begin{example}[A tame Grasstope]\label{eg:tame} 
    Let $$Z = \begin{bmatrix} 1 & 0 & -1 & -3 & -2\\ 0 & 1 & 1 & 2 & 1\\ 0 & 0 & 1 & -1 & -2\end{bmatrix}^T.$$
    This matrix is not totally positive, since $p_{123} = 1$ and $p_{124} = -1$. However, the first two rows of $Z^T$ span a totally positive line, so $Z$ satisfies \eqref{cond25} with the matrix $M$ being 
    $$\begin{bmatrix}
        1 & 0 & 0\\
        0 & 1 & 0
    \end{bmatrix}^T.$$
Therefore the resulting Grasstope is tame but is not an amplituhedron. 
The rows of $Z$ define $5$ linear forms, as noted in Remark \ref{rem:twistor}:
$$l_1 = z, \quad l_2 = -y, \quad l_3 = x-y-z,\quad l_4 = -x-2y-3z,\quad l_5 = -2x-y-2z.$$

Not every affine chart that we choose results in a bounded picture. For instance, if we map $(x:y:z) \mapsto (x+y: y: z)$ and dehomogenize with respect to the first coordinate, the resulting picture is unbounded. However, as predicted by Corollary \ref{cor:tame}, there are lines disjoint from the Grasstope. One of them is $\{-4x+z=0\}$ in $\bP^2$. By picking it to be the line at infinity (that is, by mapping $(x:y:z)\mapsto(-4x+z:y:z)$ and dehomogenizing with respect to the first coordinate), we obtain affine lines given by the linear forms
$$\Tilde{l}_1 = \Tilde{y}, \quad \Tilde{l}_2 = -\Tilde{x}, \quad \Tilde{l}_3 = \frac{-1-4\Tilde{x}-3\Tilde{y}}{4},\quad \Tilde{l}_4 = \frac{1-8\Tilde{x}-13\Tilde{y}}{4},\quad \Tilde{l}_5 = \frac{1-2\Tilde{x}-5\Tilde{y}}{2}.$$ 

Each line has an orientation, with the positive half-space given by the points $(\Tilde{x}, \Tilde{y})$ for which $\Tilde{l}(\Tilde{x}, \Tilde{y}) \geq 0$. The lines divide the affine plane into regions, each of which has a corresponding sign vector with $i^{th}$ coordinate being $+$ if $\Tilde{l}_i(\Tilde{x}, \Tilde{y})> 0$ for all $(\Tilde{x}, \Tilde{y})$ in the region, and $-$ if $\Tilde{l}_i(\Tilde{x}, \Tilde{y})< 0$.
The Grasstope of $Z$ consists exactly of those points in the regions for which $\overline{\var}(u) \geq 2$, as can be seen in Figure \ref{fig:tame2}.
\begin{figure}
    \centering
    \includegraphics[scale=0.3]{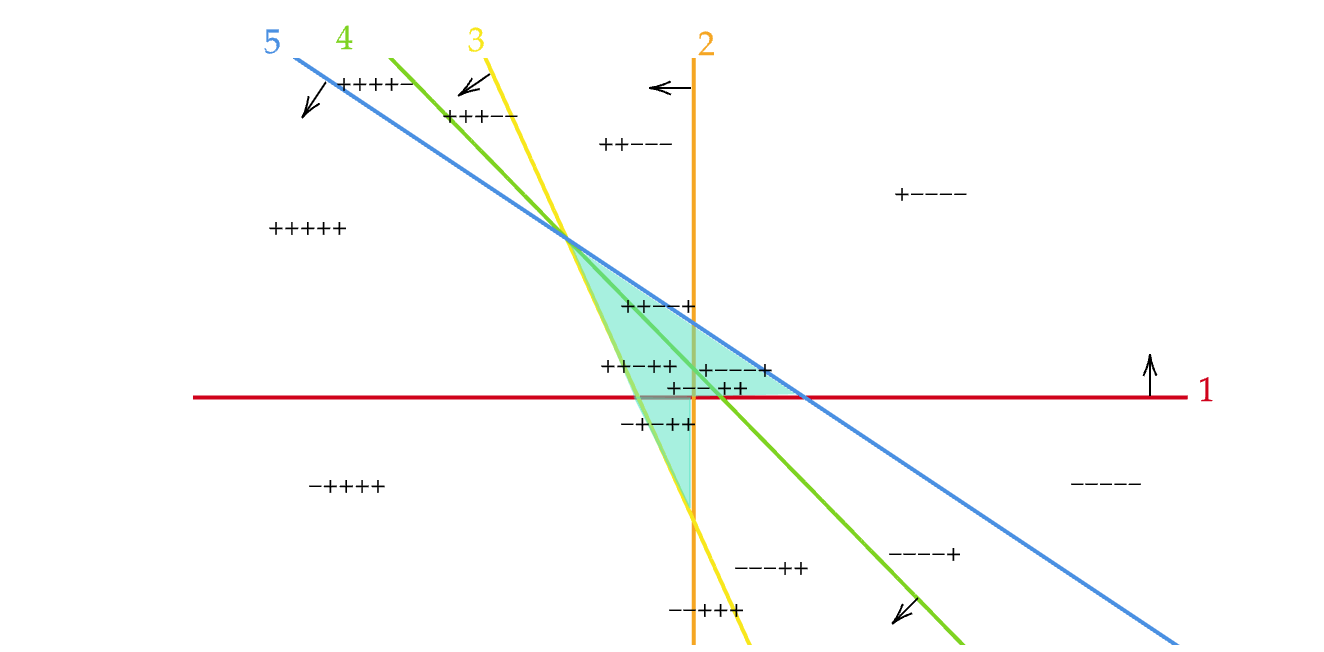}
    \caption{Affine chart in which the tame Grasstope is bounded. The six lines corresponding to the rows of $Z$ are colored red, orange, yellow, green, and blue, in order, with orientations given by arrows. The shaded portion of the figure is the Grasstope, which consists exactly of the regions with at least two sign changes.}
    \label{fig:tame2}
\end{figure}
\end{example}

\begin{example}[A wild Grasstope]\label{eg:wild}
Let $$Z = \begin{bmatrix}  2 & 2 & 0 & -1 & 1 & 0\\ 2 & 3 & 1 & 0 & 2 & 0\\ 1 & 2 & 0 & 0 & 2 & 1\end{bmatrix}^T.$$
This is the example found by Galashin (and communicated to us by Lam \cite{LamLetter}) to show that wild Grasstopes exist. Indeed, suppose there exists a $3 \times 2$ matrix $M$ such that $ZM$ has positive $2 \times 2$ minors. The $2 \times 2$ minors of $ZM$ are the entries of $\wedge^2(ZM) = \wedge^2(Z)\times \wedge^2(M) = $
\[ = \begin{bmatrix}
 2 &2 & 2&  2& 3& 1 &2& 1& -1 &-2 & 0 & 0& 0& 0&  0 \\
      2&  0&  0&  1& 2& 0& 3& 2& 0 & -2& 2& 2& 0& -1& 1 \\
 1& -1& -2& 0& 0& 0 &2 &2& 2 & 0 & 2 &3&1& 0&  2 
\end{bmatrix}^T \begin{bmatrix}
    p_{12}\\ p_{13}\\ p_{23}
\end{bmatrix},\]
where $p_{12}, p_{13}, p_{23}$ are the minors of $M$. The entries of $(\wedge^2 Z)^T$ are written here with respect to the reverse lexicographic order: column 3 corresponds to the $2\times2$ minors of the submatrix of $Z^T$ obtained by taking columns 2 and 3, while column 10 of $(\wedge^2 Z)^T$ corresponds to columns 4 and 5. Now observe that, if we suppose the $2\times 2$ minors of $ZM$ are positive, then column 3 of $(\wedge^2 Z)^T$ tells us that $p_{12} - p_{23} > 0$, column 10 tells us that $p_{12}+p_{13} < 0$, while column 11 tells us that $p_{13}+p_{23} > 0$. All three cannot be true, so no such matrix $M$ can exist.

The six rows of $Z$ correspond to the six linear forms
\begin{align*}
    l_1 &= x-2y+2z,& l_2 &= 2x-3y+2z, &l_3 &= -y\\
    l_4 &= -z,&l_5 &= 2x-2y+z,& l_6 &= x.
\end{align*}
Mapping $(x:y:z) \mapsto (x+y: y: z)$ and dehomogenizing with respect to the first coordinate, we obtain affine lines given by the linear forms
\begin{align*}
    \Tilde{l}_1 &= 2\Tilde{y}-3\Tilde{x}+1, &\Tilde{l}_2 &= 2\Tilde{y}-5\Tilde{x}+2, &\Tilde{l}_3 &= -\Tilde{x},\\
    \Tilde{l}_4 &= -\Tilde{y}, &\Tilde{l}_5 &= \Tilde{y}-4\Tilde{x}+2, & \Tilde{l}_6 &= -\Tilde{x}+1.
\end{align*}
We draw these in the affine plane and color them (in order) red, orange, yellow, green, blue, and purple. We also give the lines orientations with the positive half-space given by the points $(\Tilde{x}, \Tilde{y})$ for which
$\Tilde{l}(\Tilde{x}, \Tilde{y}) > 0$. Then the Grasstope of $Z$ consists exactly of those points in the regions between the lines for which $\overline{\var}(u) \geq 2$,  as can be seen in Figure \ref{fig:wild1}.

\begin{figure}[htbp]
    \centering
\includegraphics[scale=0.35]{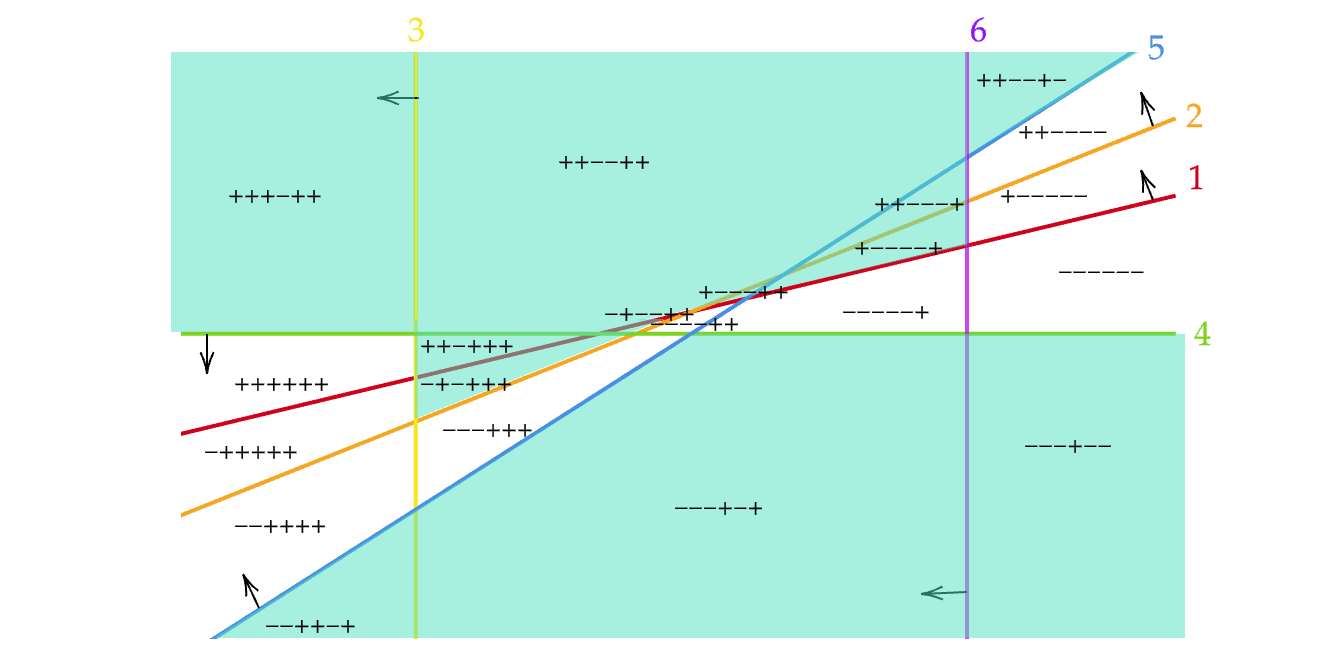}
    \caption{The six lines corresponding to the rows of $Z$ are pictured, with orientations given by the arrows. The regions can then be labelled by sign patterns. The shaded portion of the figure is the Grasstope, and it consits exactly of those regions with at least two sign changes.}
    \label{fig:wild1}
\end{figure}
\end{example}

\begin{example}[A rational Grasstope with closed boundary and M\"obius strip topology]\label{{eg:ratlclosed}}

Let \[Z = \begin{bmatrix}
   1 & 0 & 0 & -1 & 0 & 0\\
   0 & 1 & 0 & -1 & 1 & -1\\
   0 & 0 & 3 & 0 & -2 & -1
\end{bmatrix}^T.
\] 

Note that $(1, 1, 1, 1, 1, 1) \in \ker(Z^T)$, so the map $\Tilde{Z}$ has base points on $\Gr_{\geq 0}(2, 6)$. Following Proposition \ref{thm:opengrass} we can still describe its open Grasstope. As in the previous examples, we find 6 dehomogenized linear forms corresponding to six affine lines
\begin{align*}
    \Tilde{l}_1 &= \Tilde{y},&
    \Tilde{l}_2 &= -\Tilde{x},&
    \Tilde{l}_3 &= -3\Tilde{x}+3,\\
    \Tilde{l}_4 &= \Tilde{x}-\Tilde{y},&
    \Tilde{l}_5 &= \Tilde{x}-2,&
    \Tilde{l}_6 &= 2\Tilde{x}-1.
\end{align*}
Then we can find the open rational Grasstope of $Z$ as in Figure \ref{fig:mobius}.

\begin{figure}
    \centering
    \includegraphics[scale=0.35]{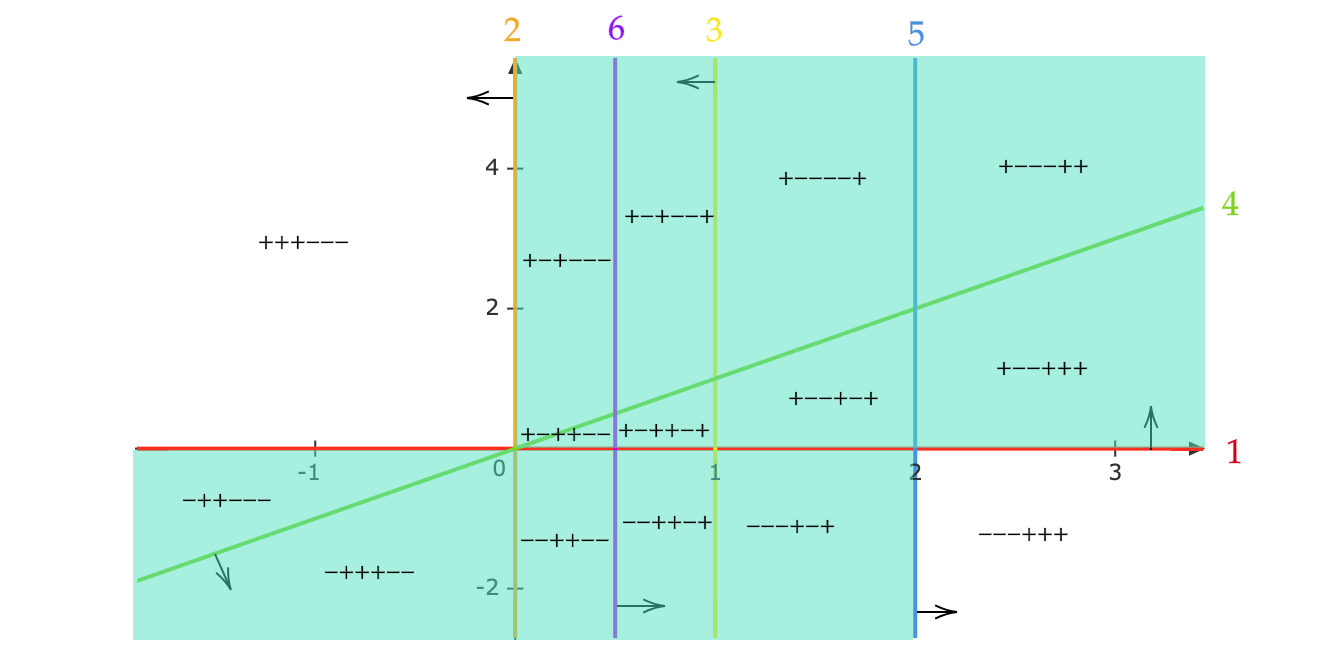}
    \caption{The six lines corresponding to the rows of $Z$ are pictured as described, with orientations given by the arrows. The regions can then be labelled by sign patterns. The shaded portion of the figure is the Grasstope, and it consits exactly of those regions with at least two sign changes. In this case, the shaded region is a M\"obius strip.}\label{fig:mobius}
\end{figure}

We claim that the rational Grasstope of $Z$ is the closure of the open rational Grasstope. One can check this by directly finding Pl\"ucker coordinates for some $M \in \Gr_{\geq 0}(2, 6)$ which map to the points that lie on the boundary. For example, to find such $M$ for any point of the form $(0, a)$ with $a\geq 0$, we solve $\wedge^2(Z)\times \wedge^2(M) = (1:0:a)$. Then $\wedge^2(M)$ must have nonnegative entries which satisfy the Pl\"ucker relations. This process is made easier if one recalls that the Pl\"ucker relations are trivially satisfied if all entries are zero except for ones which correspond to pairwise overlapping submatrices, that is any two nonzero minors come from submatrices which share a column. In this case, since \[\wedge^2(Z) = \begin{bmatrix}
    1 &0 &0 &-1& 1& 0& 1&  0&  0&  -1& -1& 0&  0& 1& 0 \\ 
     0& 3& 0 &0&  0& 3& -2& 0&  0&  2&  -1& 0 & 0 &1& 0\\  
      0& 0& 3& 0&  0& 3& 0&  -2& -3& 2&  0 & -1& 3& 1& -3
\end{bmatrix}^T,\]
we find that the point $(1:0:a)$ is given by $\wedge^2(M) = (0:0:a/3:0:1:0:0:0:0:0:0:0:0:0:0)$. 
Since the only nonzero entries correspond to the minors $p_{23}$ and $p_{24}$, the Pl\"ucker relations are satisfied. 
Thus the portion of the boundary line $x=0$ with $y \geq 0$ is part of the rational Grasstope of $Z$. One can similarly check that all other parts of all six lines are included. 
Therefore the rational Grasstope of $Z$ is closed. Furthermore, topologically, as one can see from Figure \ref{fig:mobius}, it is a M\"obius strip. 
\end{example}

\section{Background on Oriented Matroids}\label{sec:ormatdefs}

Much of the material on hyperplane arrangements can be naturally generalized to oriented matroids \cite{orientedmatroids}. In this section, we review the basics of oriented matroid theory and recall a dictionary between hyperplane arrangements and oriented matroids. This will prove useful in algorithmically counting the number of regions in a given Grasstope (see Section \ref{sec:top}). We begin by discussing signed circuits, which we will often abbreviate as circuits if the meaning is clear from context.

\begin{definition}[Signed circuit axioms]\label{def:circuits}
An oriented matroid consists of a ground set $\mathcal{E}$ and a collection $\mathcal{C}$ of tuples of the form $X = (X^+, X^-)$ called \emph{circuits}, where $X^+, X^-$ are disjoint subsets of $\mathcal{E}$ satisfying
   \begin{enumerate}
    \item $\emptyset$ is not a circuit.
    \item If $X$ is a circuit, then so is $-X = (X^-,X^+)$.
    \item No proper subset of a circuit is a circuit.
    \item (Elimination). If $X_0$ and $X_1$ are two signed circuits with $X_0 \neq -X_1$ and $e \in X_0^+ \cap X_1^-,$ then there is a third circuit $X \in \cC$ with $X^+ \subset (X_0^+ \cup X_1^+) \setminus \{e\}$ and $X^- \subset (X_0^- \cup X_1^-) \setminus \{e\}.$
\end{enumerate} 
\end{definition}

We can obtain the oriented matroid of a matrix as follows.

\begin{definition}[Oriented matroid of a matrix] \label{def:ormatmat}
    Fix a matrix $A$ and let $  \sum_i \lambda_i v_i$ be a minimal linear dependency among its rows. Associate to this dependency the signed set $X = (X^-, X^+),$ where 
    \begin{align*}
        X^- = \{i : \lambda_i < 0\}\\
        X^+ = \{i : \lambda_i > 0\}.
    \end{align*}
    Then the \emph{oriented matroid $\cM_A$} associated to $A$ has as its signed circuits the signed sets coming from minimal linear dependencies. 
\end{definition}

One can check that oriented matroids of matrices satisfy the signed circuit axioms.

\begin{example}[Matroid of a matrix $A$] \label{ex:ormatmat}
Consider the matrix
$$A = \begin{bmatrix}
    1 & 0 & 0 \\
    0 & 1 & 0\\
    0 & 0 & 1\\
    2 & -3 & 4
\end{bmatrix}.$$
Then the only linear dependency up to scaling is $v_4-4v_3 + 3v_2 - v_1 = 0.$ Thus the only circuit (when only one of $X, -X$ is considered) is $24\overline{13},$ where we use a more compact notation in which the bar indicates being in the negative part of the signed set.
\end{example}

\begin{remark}
    One can also define $\cM_A$ by signed bases; the matroid is the map which assigns to each size $k$ subset $I \subset [n]$ the sign of the determinant of $A_I.$ This definition is called the \emph{chirotope definition} \cite[page 6]{orientedmatroids} and satisfies the \emph{chirotope axioms,} which we do not describe here. In particular, the database \cite{database} used in Section \ref{sec:top} indexes matroids by chirotope.
\end{remark}

We recall a few definitions we will need to explain how a hyperplane arrangement can be viewed as an oriented matroid.

\begin{definition}[Composition]
Let $X = (X^+, X^-)$ and $Y = (Y^+, Y^-)$ be signed sets. Then their \emph{composition} $X \circ Y$ is $(X^+ \cup (Y^+ \setminus X^-), X^- \cup (Y^- \setminus X^+)).$
\end{definition}

\begin{definition}[Orthogonality]
Let $X = (X^+, X^-)$ and $Y = (Y^+, Y^-)$ be signed sets. Define $S(X, Y) = (X^+ \cap Y^-) \cup (X^- \cap Y^+).$ We say $X$ and $Y$ are \emph{orthogonal} if $S(X, Y)$ and $S(X, -Y)$ are both empty or both non-empty.
\end{definition}

We define \emph{vectors} as compositions of circuits, and \emph{cocircuits} and \emph{covectors} as the circuits and vectors of the dual matroid \cite[page 4]{orientedmatroids}, respectively. An equivalent definition which is easier for computing is that the covectors of $\cM$ are the signed sets which are orthogonal to all vectors of $\cM$. For more detail, see \cite[Chapter 1]{orientedmatroids}. There is yet another definition in the case of vector configurations.

\begin{definition}[Oriented matroid of a vector configuration]
One can also view the rows $v_i$ of the matrix $A$ as vectors in $\bR^k.$ For such a vector configuration, the covectors can be defined as the set of tuples $Y_H = (Y_H^+, Y_H^-)$ as $H$ runs over oriented hyperplanes, where $Y_H^+$ is the set of vectors in the positive halfspace defined by $H$, and $Y_H^-$ is the set of vectors in the negative halfspace. The \emph{cocircuits} are the minimal covectors. They arise from hyperplanes that are spanned by subsets of $\{v_1, ..., v_n\}.$
\end{definition}

We call a hyperplane arrangement in $\mathbb{A}^k$ \emph{central} if all of the hyperplanes pass through the origin. We define the oriented matroid of a central arrangement as follows. 

\begin{definition}[Oriented matroid of a central hyperplane arrangement]
    Let $H_i$ be the hyperplane in $\mathbb{A}^k$ given by the vanishing of $l_i(x) = a_i \cdot x.$ Then the oriented matroid of the arrangement $\{H_i\}_{i=1}^n$ is the matroid of the vector configuration given by $a_1, ..., a_n,$ or equivalently, the oriented matroid of the matrix with rows $a_1, ..., a_n$. Note that the matroid takes as data not just the hyperplanes, but a choice of ``positive direction" for each hyperplane, which is recorded by choosing $a_i$ rather than $-a_i$ in the linear form $l_i(x).$
\end{definition}

\begin{remark}\label{rem:dict}
   Faces of a hyperplane arrangement correspond to covectors of its oriented matroid, and regions correspond to maximal covectors. The \emph{rank} (denoted by $r$) of the oriented matroid is equal to the dimension $k$ of the ambient space \cite[Chapter 1]{orientedmatroids}.  
\end{remark}

\begin{example}[Matroid of a matrix $A$]
Let $A$ be the matrix 
$$A = \begin{bmatrix}
    1 & 0 & 0 \\
    0 & 1 & 0\\
    0 & 0 & 1\\
    2 & -3 & 4
\end{bmatrix}$$
as in Example \ref{ex:ormatmat}. The cocircuits are given considering hyperplanes $H$ spanned by pairs of rows. For example, if $H = \text{Span}\{a_1, a_3\} = \{y = 0\},$ then $a_2$ is in the positive halfspace (since the $y$-coordinate is $1$) and $a_4$ is in the negative halfspace (since the $y$ coordinate is $-3$). Thus we obtain the cocircuit $2\bar{4}$.

The total set of cocircuits (not including negations) is $\{2\bar{4}, 34, 23, 14, 3\bar{1}, 12\}$. In the previous sections we have used sign vector notation; that is, to each signed set we associate a length $n$ vector with $\pm$ at index $i$ if $i$ is in $X^\pm$ and $0$ otherwise. Applying this convention, one can check that the covectors are exactly the sign vectors with fewer than $3$ sign changes. 
For example, $2\bar{4} \circ 3\bar{1} = 23\overline{14},$ which corresponds to $(-++-).$
\end{example}

\begin{definition}
A matroid is \emph{realizable} if it arises from a hyperplane arrangement over $\bR$. 
\end{definition}

\section{Extremal Counts and Oriented Matroid Grasstopes} \label{sec:top}

Given an $n\times (k+1)$ matrix $Z$, one may ask questions about the topology of its $m=1$ Grasstope. For instance, is it closed, connected, contractible? How many regions of the hyperplane arrangement does it contain? For the $m = 1$ amplituhedron, the answer to the first three questions is ``yes'' \cite[Corollary 6.18]{KarpWilliams}. As for the latter, the $m = 1$ amplituhedron contains as few regions as possible for a simple hyperplane arrangement; that is, all sign vectors with $\var<k$ appear as labels of regions in the arrangement \cite[Proposition 6.14]{KarpWilliams}. In this section, we investigate this question for more general $m=1$ Grasstopes.

\begin{definition}
    A \emph{region} of a hyperplane arrangement is a connected component of the complement of the union of hyperplanes in the arrangement. 
\end{definition}

\begin{definition}
    Let $\mathcal{A}$ be an arrangement of $n$ hyperplanes in $\mathbb{A}^{k}$. Then $\mathcal{A}$ is called \emph{simple} if the intersection of any $i$ hyperplanes in $\mathcal{A}$ has codimension $i$ in $\mathbb{A}^k$ for all $i \leq k$, and is empty for $i > k$.
\end{definition}

In \cite{zaslavsky}, Zaslavsky gives the following formulae for the number of total regions $t(\cA)$ and bounded regions $b(\cA)$ of a simple affine arrangement $\cA$ of $n$ hyperplanes in $\mathbb{A}^k$:

\begin{align*} \label{eq:affregions}
    t(\cA) & = 1 + n + \binom{n}{2} + ... + \binom{n}{k}, \\
    b(\cA) & = \binom{n-1}{k}.
\end{align*}

The former count first appeared in work of Buck \cite{buck}. Note that since intersections of two hyperplanes in any projective arrangement $\mathcal{P}$ have codimension two in $\mathbb{P}^k$ and there are a finite number of them, we can always find some hyperplane avoiding them, and hence an an affine chart which contains all of them (see Figure \ref{fig:threeplanes}). A projective arrangement $\mathcal{P}$ naturally induces an affine arrangement $\mathcal{A}$ in this affine chart. We call $\mathcal{P}$ \emph{simple} if this corresponding affine arrangement is simple. If a region of $\mathcal{P}$ intersects the chosen hyperplane at infinity, it induces two unbounded regions of $\mathcal{A}$. Otherwise it induces a single bounded region. Thus, the total number of regions in a simple arrangement $\mathcal{P}$ is

\begin{equation} \label{eq:projregions}
    r(\mathcal{P}) = b(\cA) + \frac{t(\cA) - b(\cA)}{2}.
\end{equation}

\begin{figure}[htbp]
\centering

\tikzset{every picture/.style={line width=0.75pt}} 

\begin{tikzpicture}[x=0.75pt,y=0.75pt,yscale=-1,xscale=1]

\draw  [line width=1.5]  (187,142.5) .. controls (187,73.19) and (243.19,17) .. (312.5,17) .. controls (381.81,17) and (438,73.19) .. (438,142.5) .. controls (438,211.81) and (381.81,268) .. (312.5,268) .. controls (243.19,268) and (187,211.81) .. (187,142.5) -- cycle ;
\draw [color={rgb, 255:red, 229; green, 35; blue, 35 }  ,draw opacity=1 ][line width=1.5]    (399,52) -- (202,201) ;
\draw [color={rgb, 255:red, 208; green, 2; blue, 27 }  ,draw opacity=1 ]   (383,63) -- (373.24,50.57) ;
\draw [shift={(372,49)}, rotate = 51.84] [color={rgb, 255:red, 208; green, 2; blue, 27 }  ,draw opacity=1 ][line width=0.75]    (10.93,-3.29) .. controls (6.95,-1.4) and (3.31,-0.3) .. (0,0) .. controls (3.31,0.3) and (6.95,1.4) .. (10.93,3.29)   ;
\draw [color={rgb, 255:red, 74; green, 144; blue, 226 }  ,draw opacity=1 ][line width=1.5]    (270,24) -- (378,248) ;
\draw [color={rgb, 255:red, 74; green, 144; blue, 226 }  ,draw opacity=1 ]   (274,38) -- (268.41,41.35) -- (260.71,45.97) ;
\draw [shift={(259,47)}, rotate = 329.04] [color={rgb, 255:red, 74; green, 144; blue, 226 }  ,draw opacity=1 ][line width=0.75]    (10.93,-3.29) .. controls (6.95,-1.4) and (3.31,-0.3) .. (0,0) .. controls (3.31,0.3) and (6.95,1.4) .. (10.93,3.29)   ;
\draw [color={rgb, 255:red, 126; green, 211; blue, 33 }  ,draw opacity=1 ][line width=1.5]    (189,125) -- (420,207) ;
\draw [color={rgb, 255:red, 126; green, 211; blue, 33 }  ,draw opacity=1 ]   (211,133) -- (203.81,149.17) ;
\draw [shift={(203,151)}, rotate = 293.96] [color={rgb, 255:red, 126; green, 211; blue, 33 }  ,draw opacity=1 ][line width=0.75]    (10.93,-3.29) .. controls (6.95,-1.4) and (3.31,-0.3) .. (0,0) .. controls (3.31,0.3) and (6.95,1.4) .. (10.93,3.29)   ;

% Text Node
\draw (370,123.4) node [anchor=north west][inner sep=0.75pt]  [font=\small]  {$---$};
% Text Node
\draw (370,204.4) node [anchor=north west][inner sep=0.75pt]  [font=\small]  {$--+$};
% Text Node
\draw (267,199.4) node [anchor=north west][inner sep=0.75pt]  [font=\small]  {$-++$};
% Text Node
\draw (284,141.4) node [anchor=north west][inner sep=0.75pt]  [font=\small]  {$-+-$};
% Text Node
\draw (234,97.4) node [anchor=north west][inner sep=0.75pt]  [font=\small]  {$++-$};
% Text Node
\draw (304,51.4) node [anchor=north west][inner sep=0.75pt]  [font=\small]  {$+--$};
% Text Node
\draw (212,151.4) node [anchor=north west][inner sep=0.75pt]  [font=\small]  {$+++$};
% Text Node
\draw (405,35.4) node [anchor=north west][inner sep=0.75pt]  [font=\normalsize,color={rgb, 255:red, 208; green, 2; blue, 27 }  ,opacity=1 ]  {$1$};
% Text Node
\draw (250,10.4) node [anchor=north west][inner sep=0.75pt]  [font=\normalsize,color={rgb, 255:red, 208; green, 2; blue, 27 }  ,opacity=1 ]  {$\textcolor[rgb]{0.29,0.56,0.89}{2}$};
% Text Node
\draw (167,116.4) node [anchor=north west][inner sep=0.75pt]  [font=\normalsize,color={rgb, 255:red, 208; green, 2; blue, 27 }  ,opacity=1 ]  {$\textcolor[rgb]{0.49,0.83,0.13}{3}$};

\end{tikzpicture}
    \caption{An oriented hyperplane arrangement and its sign labels.}
    \label{fig:threeplanes}
\end{figure}

From here on we assume our arrangements are simple. We write $\beta(k,n)$ for the number of possible sign patterns of length $n$ with sign variation less than $k$ and $\gamma(k,n)$ for the number of sign patterns with variation greater or equal than $k$ (we identify sign patterns $\sigma$ and $-\sigma$). Note that $\beta(k,n) = 1 + (n-1) + \binom{n-1}{2} + ... + \binom{n-1}{k},$ and $\gamma(k,n) = 2^{n-1}-\beta(k,n).$ Theorem \ref{thm:welldefgrass} and Proposition \ref{prop:opengrass} then give the lower bound $r(\mathcal{P}) - \beta(k,n)$ for the number of regions in $\mathcal{G}_{n,k,1}(Z)$, where $\mathcal{P}$ is the hyperplane arrangement defined by $Z$. An upper bound is given by the minimum of $\gamma(k,n)$ and $r(\mathcal{P})$.

The database \cite{database} contains a catalog of isomorphism classes of oriented matroids \cite[Section 6]{finschiPhD}. Each matroid is indexed by a vector of signs of its bases. Recall from the previous section that any central arrangement produces an oriented matroid. Each projective arrangement $\mathcal{P} \subset \bP^{k}$ produces an oriented matroid: we take the cone over $\mathcal{P}$ to get a central arrangement in $\mathbb{A}^{k+1}$. Observe that an arrangement is simple if and only if the vector of signs of bases of its matroid does not contain zeros, that is, the underlying (non-oriented) matroid is uniform.  

We iterate over all uniform oriented matroid isomorphism classes in this catalog for small values of $k$ and $n$. Note that, for all the values of $k$ and $n$ which we consider, all uniform matroids are realizable \cite{fukuda}, that is, arise from hyperplane arrangements.
Within each isomorphism class, we iterate over all possible reorderings of the ground set and reorientations. 
If the matroid is realizable, at the level of the matrix $Z$ defining the arrangement as described in Section \ref{sec:prelim}, reorderings correspond to permuting the rows and reorientations to negating certain rows.
For each isomorphism class, ordering of the hyperplanes, and a choice of orientation, we compute the number of regions in the corresponding Grasstope (that is, the number of maximal covectors with sign variation greater or equal than $k$; see Remark \ref{rem:dict}). 
It turns out that for many values of $k$ and $n$ the minimal and maximal number of regions in the Grasstope when iterating over reorderings and reorientations does not depend on the oriented matroid isomorophism class.
The minimal and maximal number of regions in the Grasstope for these values of $k$ and $n$ are presented in Table \ref{table:regiondata}.
The \texttt{Python} code used to extract this data is available at \url{https://mathrepo.mis.mpg.de/Grasstopes}. We therefore have a computational proof of the following statement. 

\begin{proposition}
    For each pair of values of $k$ and $n$ in Table \ref{table:regiondata} the minimal and maximal possible number of regions in a Grasstope arising from a simple arrangement of $n$ hyperplanes in $\mathbb{P}^k$ do not depend on the choice of arrangement.
\end{proposition}

Out of the entries in Table \ref{table:regiondata}, note that for $k=2$ and the pairs $(3,5)$ and $(4,6),$ there is only one (uniform) oriented matroid up to isomorphism \cite{database}. For $k=2$, this is because any matrix may be turned into a totally positive matrix by permuting the rows. This can be done by viewing rows as vectors in the plane and arranging them in counterclockwise position.

\begin{table}[h!] 
\centering
\begin{tabular}{ | m{1.5cm} | m{2cm} | m{2cm}| m{1.5cm}| m{1.5cm}|m{1.5cm}|} 
  \hline
  $k$, $n$ & Minimal & Maximal&  $r(\mathcal{P})$ & $\beta(k,n)$ & $\gamma(k,n)$\\
  \hline 
  $2$, $6$ & $10$ & $16$ & $16$ & $6$ & $26$\\
  $2$, $7$ & $15$ & $22$ & $22$ & $7$ & $57$\\
%  2, 8 **& 21 & 29 & 29 & 8 &\\
  $3$, $5$ & $4$ & $5$ & $15$ & $11$ & $5$\\
  $3$, $6$ & $10$ & $16$ & $26$ & $16$ & $16$\\
  $4$, $6$ & $5$ & $6$ & $31$ & $26$ & $6$\\
  $4$, $7$ & $15$ & $22$ & $57$ & $42$ & $22$\\
  \hline
\end{tabular}
\caption{Minimal and maximal possible number of regions in a Grasstope.}
\label{table:regiondata}
\end{table}

The pair $(3, 7)$ does not appear in Table \ref{table:regiondata}. This is the first time we see variation depending on which simple arrangement we choose, with the maximal number of regions ranging from $38$ to $42$. The reorientations and reorderings of a totally positive matrix (i.e. the amplituhedron case) give at most $42$ regions, while all other oriented matroid classes achieve fewer. We can see the maximal numbers of regions for other small $k, n$ in Table \ref{table:posdata}.

\begin{table}[h!] 
\centering
\begin{tabular}{ | m{2cm} | m{2cm} | m{2cm}| m{2cm}|} 
  \hline
  $k$, $n$ & Maximal&  $r(\mathcal{P})$ & $\gamma(k,n)$\\
  \hline 
  $3$, $7$ & $42$  & $42$ & $42$ \\
  $3$, $8$ & $64$ & $64$ & $99$ \\
  $4$, $8$ & $64$ & $99$ & $64$ \\
  $5$, $8$ & $29$ & $120$ & $29$ \\
  $2$, $9$ & $37$ & $37$ & $247$ \\
  $3$, $9$ & $93$ & $93$ & $219$ \\
  $4$, $9$ & $163$ & $163$ & $163$ \\
  \hline
\end{tabular}
\caption{Maximal number of regions from reorienting and reordering a positive matrix.}
\label{table:posdata}
\end{table}

\begin{example}
Any totally positive $6 \times 3$ matrix with the second and fourth rows negated yields a Grasstope which includes all 16 regions counted by Equation \eqref{eq:projregions}. The resulting hyperplane arrangement is cyclic, with just two orientations flipped. See Figure \ref{fig:allregions} to see all of the regions labelled with sign patterns.

 \begin{figure}[htbp]
     \centering
\includegraphics[scale=0.3]{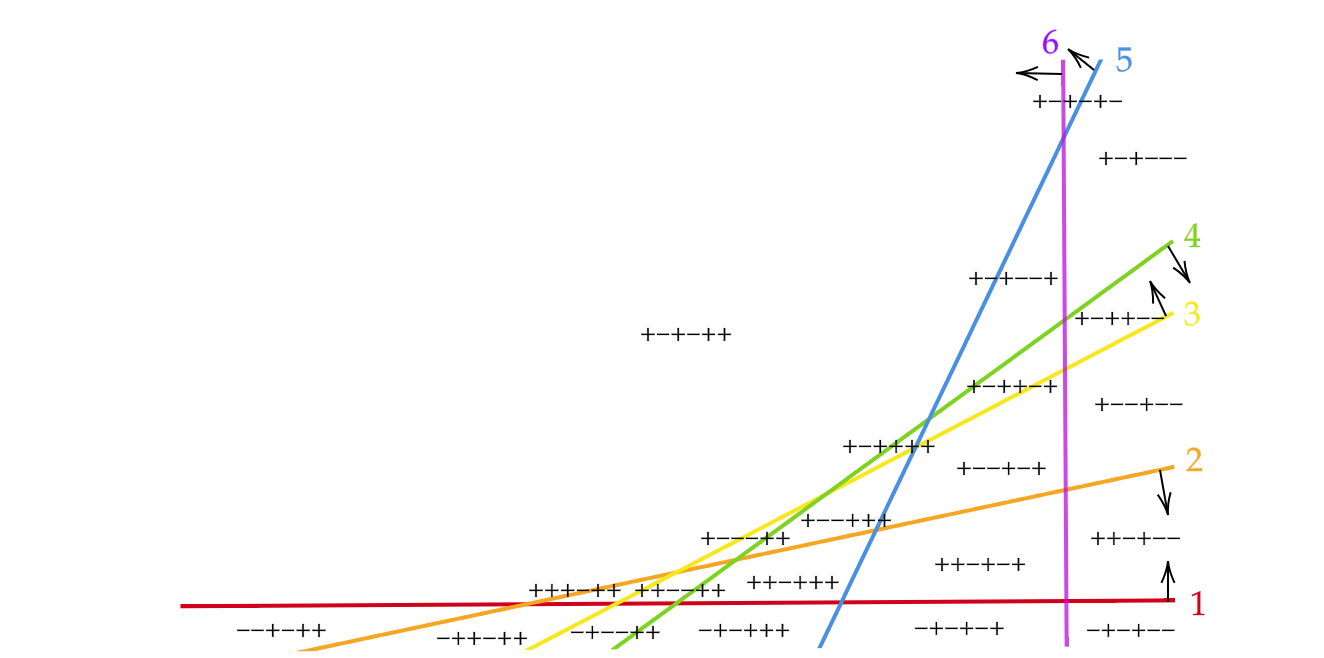}
     \caption{The Grasstope of a totally positive matrix with two rows negated. The six lines are cyclically ordered with orientations indicated by arrows. Every region has at least two sign changes, so the Grasstope is all of $\bP^2$.}
     \label{fig:allregions}
 \end{figure}   
%\vspace{-1cm}
\end{example}

\begin{example}
An example of a $6 \times 3$ matrix whose Grasstope has $16$ regions is any totally positive matrix with the $2$nd and $4$th rows swapped. For examples of totally positive matrices, one can take the Vandermonde matrix 
\[\begin{bmatrix}
    1 & 1 & 1 & 1 & 1 & 1\\
    d_1 & d_2 & d_3 & d_4 & d_5 & d_6\\
    d_1^2 & d_2^2 & d_3^2 & d_4^2 & d_5^2 & d_6^2
\end{bmatrix}^T.\]
with $0<d_1 < ... < d_6.$
\end{example}

Note that the lower bound $r(\mathcal{P}) - \beta(k,n)$ for the number of regions in a Grasstope is attained for all $k$ and $n$, by any tame $m=1$ Grasstope whose arrangement is simple \cite[Proposition 6.14, Theorem 6.16]{KarpWilliams}. In particular, it is attained by the $m=1$ amplituhedron. The upper bound is also attained by reorientation class of the $m=1$ amplituhedron for the values in Tables \ref{table:regiondata} and \ref{table:posdata}. One interesting question to study is to determine whether this holds in general, and to describe all oriented matroids achieving this upper bound. 

\begin{problem}\label{prob:maxregions}
    For any oriented matroid $\cM,$ Let $f(\cM)$ be the maximum over the reorientations of $\cM$ of $\#\{\text{covectors with sign variation greater than or equal to } k\}.$ Let $\cM_{pos}$ be the oriented matroid of a totally positive matrix.  Is $f(\cM_{pos})$ equal to the minimum of $\gamma(k, n)$ and $r(\mathcal{P})$? 
\end{problem}

\begin{problem}
    Which oriented matroids achieve the upper bound in Problem \ref{prob:maxregions}?
\end{problem}

The dictionary between hyperplane arrangements and oriented matroids (Remark \ref{rem:dict}) guides us to generalize the definition of an $m=1$ Grasstope to oriented matroids that are not necessarily realizable, such that the definitions agree when $\cM = \cM_Z$.
 
\begin{definition}[Grasstope of an oriented matroid]\label{def:matroidgtope}
Let $\cM$ be an oriented matroid of rank $r$, and $<$ be a total order on the ground set $\mathcal{E}$ of $\cM.$ Then we define the \emph{Grasstope $\cG(\cM, <)$} to be the subset of covectors $\{v : \overline{\var}(v) \geq r-1\},$ where $r$ is the rank of $\cM$ and the variation is with respect to $<$.
\end{definition}

Note that by Topological Representation Theorem \cite[Theorem 20]{FolLawr}, every oriented matroid arises from a pseudohyperplane arrangement, with covectors labelling the cells of this arrangement. In particular, the Grasstope $\cG(\cM, <)$ can be identified with the union of cells of a pseudohyperplane arrangement that satisfy the sign variation condition from Definition \ref{def:matroidgtope}. Therefore, Grasstopes of oriented matroids are meaningful geometric objects, and topological concepts such as connectedness and contractibility generalize naturally to them. Studying their topological properties is an interesting topic for future research.  

Finally, we use our code to analyze a non-realizable example, which does not attain the upper bound.

\begin{example}
Consider the non-realizable matroid FMR(8) of rank $4$ on $8$ elements, whose signed cocircuits are given in \cite[Table 1]{roudneff}. Reorientations and reorderings give at least $34$ and at most $63$ regions. Thus, unlike the amplituhedron, FMR(8) does not achieve the upper bound of~$64$.
\end{example}

 \section*{Acknowledgements}
 We thank Bernd Sturmfels for suggesting the topic and for many helpful conversations. We also thank Matteo Parisi and Thomas Lam for interesting discussions. Finally, we are very grateful to the anonymous referee for the detailed comments that improved our paper and for suggesting a simpler proof of Lemma 3.2. Y.M. was partially supported by NSF grant DGE2146752.
\bibliographystyle{abbrv}
\bibliography{sources}

\bigskip
\bigskip

\noindent
\footnotesize
 {\bf Authors' addresses:}

 \smallskip

 \noindent Yelena Mandelshtam,
 IAS
 \hfill \url{yelenam@ias.edu}

 \noindent Dmitrii Pavlov,
 MPI-MiS Leipzig and TU Dresden (current)
 \hfill \url{dmitrii.pavlov@mis.mpg.de}

 \noindent Elizabeth Pratt,
 UC Berkeley
 \hfill \url{epratt@berkeley.edu}
\end{document}